\documentclass[12pt,onesided]{article}

\usepackage[leqno]{amsmath}
\usepackage{amsthm}
\usepackage{amssymb}
\usepackage{amsfonts}
\usepackage{array}
\usepackage{hyperref}

\def \beq{\begin{equation}}
\def \eeq{\end{equation}}
\def \beqa{\begin{eqnarray}}
\def \eeqa{\end{eqnarray}}
\def \N{\mathbb{N}}

\newtheorem{conjecture}{CONJECTURE}
\newtheorem{theorem}{THEOREM}
\newtheorem{corollary}{COROLLARY}
\newtheorem{proposition}{PROPOSITION}
\newtheorem{lemma}{LEMMA}
\newtheorem{definition}{DEFINITION}

\newtheorem{example}{EXAMPLE}

\def\bi{\begin{itemize}}
\def\ei{\end{itemize}}

\begin{document}

\title{Share at least half the numbers in a nontrivial LCM-closed set a nontrivial divisor?}
\author{Tom Fischer\thanks{Institute of Mathematics, University of Wuerzburg, 
Emil-Fischer-Strasse 30, 97074 Wuerzburg, Germany.
Phone: +49 931 3188911.
E-mail: {\tt tom.fischer@uni-wuerzburg.de}.
}\\
University of Wuerzburg}
\date{\today}

\maketitle

\begin{abstract}
For a finite set of non-zero natural numbers that contains at least one element different from 1
and the least common multiple of any of its subsets, there exists a subset of at least half of its members
which has a common divisor larger than 1. 
Utilizing a representation of the natural numbers as an order-theoretical ring of prime power sets,
this conjecture is shown to be equivalent to Frankl's union-closed sets conjecture.
Some results for cases where the conjecture, which also has meaningful interpretations in graph and lattice theory, 
is known to hold are provided.
An equivalent dual version of the conjecture is, that for a finite set of non-zero natural numbers that contains at least two elements
and the greatest common divisor of any of its subsets,
one of its members has a prime power that is not a prime power of more than half of the members.
\end{abstract}





\section{Preliminaries}

For $i\in\mathbb{N}^+=\mathbb{N}\setminus\{0\}$, let $p_i$ denote the $i$-th prime number. 
Ubiquitously known as the Fundamental Theorem of Arithmetic (see \cite{CFG} for an early reference), 
there exists a uniquely determined injective function
\beqa
\label{eq:1}
q: \; 
\mathbb{N}^+ & \longrightarrow & \mathbb{N}^{\mathbb{N}^+} \\
\nonumber
n & \longmapsto & (q_1(n), q_2(n), \ldots)
\eeqa
such that
\beqa
\label{eq:2}
n & = & \prod_{i\in\mathbb{N}^+}(p_i)^{q_i(n)} .
\eeqa
Ignoring factors equal one, the right hand side of \eqref{eq:2} is called the prime factorization of $n$. 
With the exception of $q(1) = (0,0,\ldots)$, the members of $q(\N^+)$ are the finite sequences of $\N^{\N^+}$ in 
the sense that any element of $q(\N^+)$ is constant zero from certain member 
of the sequence onwards.
Thus, $q$ is a by \eqref{eq:2} uniquely determined bijection between the non-zero natural numbers and the 
finite sequences of elements of $\N$ in the thus explained sense, including the sequence of zeros only.

Also well known is that the least common multiple (LCM) of $m$ numbers $n_j\in\mathbb{N}^+$, 
$j=1,\ldots,m$, where $m\in\mathbb{N}^+\setminus\{1\}$, is given by
\beqa
\label{eq:3}
\text{lcm}(n_1,\ldots,n_m) & = & \prod_{i\in\mathbb{N}^+}(p_i)^{\max\{q_i(n_1),\ldots,q_i(n_m)\}} ,
\eeqa
and their greatest common divisor (GCD) by
\beqa
\label{eq:4}
\text{gcd}(n_1,\ldots,n_m) & = & \prod_{i\in\mathbb{N}^+}(p_i)^{\min\{q_i(n_1),\ldots,q_i(n_m)\}} .
\eeqa
For finite $\mathcal{N}\subset\N^+$ with $\#\mathcal{N} > 1$, the notation $\text{lcm}(\mathcal{N})$ and
$\text{gcd}(\mathcal{N})$ with the obvious from \eqref{eq:3} and \eqref{eq:4} derived meaning will be used.


\section{An LCM-closed sets conjecture}

In the following, the subset symbol $\subset$ includes equality.

\begin{definition}
\label{def:1}
Nonempty
$\mathcal{N}\subset\mathbb{N}^+$ is LCM-closed, respectively GCD-closed, if $m,n\in\mathcal{N}$
implies $\text{\em lcm}(m,n)\in\mathcal{N}$, respectively $\text{\em gcd}(m,n)\in\mathcal{N}$.
\end{definition}

By induction, it is fairly obvious that the properties LCM-closed,
or GCD-closed, apply to finite subsets of an LCM-closed, respectively GCD-closed, set $\mathcal{N}$
in the sense that $\mathcal{N}\subset\mathbb{N}^+$ is LCM-closed, respectively GCD-closed, if and only if
$\mathcal{M}\subset\mathcal{N}$
implies $\text{lcm}(\mathcal{M})\in\mathcal{N}$, respectively $\text{gcd}(\mathcal{M})\in\mathcal{N}$,
for any finite $\mathcal{M}$.

\begin{example}
\label{ex:0}
~
{\em 
\begin{enumerate}
\item
$\mathcal{N}=\{1,2,3,4,6,12\}$ is LCM- and GCD-closed. 
\item
$\mathcal{N}=\{2,3,4,6,12\}$ is LCM-closed, but not GCD-closed. 
\item
$\mathcal{N}=\{1,2,3,4,6,8,12\}$ is GCD-closed, but not LCM-closed. 
\end{enumerate}
}
\end{example}

\begin{example}
\label{ex:00}
{\em 
Consider a dynamical system on a nonempty set $X$ given by a map $s: X \rightarrow X$. 
A nonempty subset $A\subset X$ is periodic if
$P(A) = \{n\in\N^+\setminus\{1\}: s^n(x) = x \text{ for all } x\in A\} \neq \emptyset$,
and $P_f(A) := \min P(A)$ is then called the fundamental period of $A$.
Given $P(A) \neq \emptyset$, the set $\{P_f(B) : B \subset A\}$
is LCM-closed, and $P_f(A)$ is its maximum.
If nonempty, $\{P_f(A) : A \subset X\}$ is LCM-closed, but potentially infinite.
}
\end{example}

\begin{conjecture}
\label{con:1}
For an LCM-closed finite set of non-zero natural numbers that contains at least one element
different from 1, there exists a divisor larger than 1 for a subset of at least half of its members.
\end{conjecture}

Such a nontrivial divisor will be called abundant in the considered LCM-closed set.
In Conjecture \ref{con:1}, ``divisor larger than 1'' could obviously be replaced by ``prime factor''.

\begin{example}
\label{ex:1}
{\em 
The set $\mathcal{N}=\{1,2,3,4,6,8,12,24\}$ of eight elements is LCM-closed. 
The subset $\{2,4,6,8,12,24\}$ of six elements has the greatest common divisor 2. 
}
\end{example}

The next example shows that an abundant divisor in an LCM-closed set neither has to be a member of the set,
nor does it have to be prime. 

\begin{example}
\label{ex:4asx}
{\em 
The set $\mathcal{N}=\{6,10,14,30,42,70,210\}$ is LCM-closed, but not GCD-closed. 
The numbers 2, 3, 5, 6, 7, 10, and 14 are abundant nontrivial divisors.
}
\end{example}

The following set illustrates that not each prime factor in $\text{lcm}(\mathcal{N})$ must be 
an abundant divisor in an LCM-closed set $\mathcal{N}$.

\begin{example}
\label{ex:4basx}
{\em 
The set $\mathcal{N}=\{2,6,30\}$ is LCM-closed, but $5$ only divides element $30$.
}
\end{example}

The main result of this note will be the equivalence of Conjecture \ref{con:1} to the union-closed sets conjecture,
which -- according to \cite{BS} and despite of a plethora of articles relating to it -- has resisted proof since at least 1979.
While connections to graph theory (e.g.~\cite{K}) and lattice theory (e.g.~\cite{K} and \cite{P}) are well known (see \cite{BS} for summaries), 
no relations to number theory seem to have been drawn in the past.


\section{Frankl's union-closed sets conjecture}

\label{Frankl's union-closed sets conjecture}

\begin{definition}
\label{def:2qwe}
A family (``system'') $\mathcal{S}$ of sets is union-closed, respectively intersection-closed, if $A,B\in\mathcal{S}$
implies $A\cup B\in\mathcal{S}$, respectively $A\cap B\in\mathcal{S}$.

An order-theoretical ring is a family of sets, which is simultaneously union- and intersection-closed.
\end{definition}

Similar to Definition \ref{def:1}, 
a family $\mathcal{S}$ of sets is union-closed, respectively intersection-closed, if and only if it holds for finite 
$\tilde{\mathcal{S}}\subset\mathcal{S}$ that $\bigcup\tilde{\mathcal{S}}\in\mathcal{S}$, 
respectively $\bigcap\tilde{\mathcal{S}}\in\mathcal{S}$. Obviously, any finite order-theoretical ring
is a complete lattice (by the order of inclusion) with intersection as the meet operation and union as the join operation.\\

\noindent
\textbf{Union-closed sets conjecture.}
{\em 
For a finite union-closed family of sets with at least one nonempty member, there exists an element
shared by at least half of the member sets.\\
}

Bruhn and Schaudt \cite{BS} is a fairly recent and very thorough overview article on 
the Frankl conjecture, which is why no comprehensive summary of conditions, under which the 
conjecture is known to hold, needs to be provided here. A few selected conditions will be given below.
However, besides a simple example further down,
a few basic tools and facts regarding the conjecture should be pointed out to readers who are unfamiliar with it.

Considering a union-closed family of sets, it makes sense to identify elements that lie in exactly the same sets,
thus considering them as one single element. Since in any here considered union-closed system $\mathcal{S}$ 
the maximal set (using the order of inclusion; this set is also called the ``universe'') is finite,
elements of member sets can be identified with natural numbers. For this, let the maximal set be 
$S_{\max} = \{1,2,\ldots,m\}$, $m\in\N^+$. Now, all
members of $\mathcal{S}$ are subsets of natural numbers in $S_{\max}$, where the empty set can be
a member, as well. It is easy to show -- and widely known -- that union-closed systems that either have
a singleton member set, or a member set with only two elements, adhere to the conjecture.

\begin{example}
\label{ex:5rtu}
{\em
The following is a union-closed system with eight member sets:
\beqa
\mathcal{S} & = & \{\emptyset, \{1\} , \{1,2\} , \{1,2,3\}, \{4\} , \{1,4\} , \{1,2,4\}, \{1,2,3,4\} \} .
\eeqa
The elements 1, 2, and 4 are abundant. The universe is  $S_{\max} = \{1,2,3,4\}$. Note that, for instance,
the member sets $\emptyset$, $\{1\}$, $\{1,2\}$, $\{1,2,3\}$, and $\{4\}$ could each, or all, be removed, 
while the system would remain union-closed.
}
\end{example}

A few cases where the conjecture is known to hold are (see \cite{BS} and references therein):
\begin{enumerate}
\item
$\mathcal{S}$ has a member with only one or with only two elements.
\item
$\#S_{\max} \leq 12$. 
\item
$\#\mathcal{S} \leq 50$.
\item
$\#\mathcal{S} \geq \frac{2}{3}2^{\#S_{\max}}$.
\item
$\#\mathcal{S} \leq 2(\#S_{\max})$, if $\mathcal{S}$ is such that for any $x,y\in S_{\max}$
there exist $A,B\in\mathcal{S}$ such that $x\in A$ and $y\in B$, but $x,y\notin A\cap B$.
The system $\mathcal{S}$ is then called separating.
\end{enumerate}


\section{$\N^+$ as a ring of prime power sets}

In preparation for the proof of the main result, which establishes equivalence of Conjecture \ref{con:1} and 
the union-closed sets conjecture, this section examines how the natural numbers can be
considered as an order-theoretical ring with regard to the operations ``$\text{lcm}$'' and ``$\text{gcd}$''.
Considerations similar to the ones carried out here, but regarding natural numbers as multisets of their prime
factors, may be familiar to the reader and may indeed be considered as folklore knowledge. This remark extends
to Lemma \ref{lem:1} below. However, it is better for the purpose of this note to not use
multisets of prime factors, and instead use (proper) sets of prime powers, which -- as a slightly different setup -- 
justifies the more detailed explanations provided here.

With $\mathcal{P}(\N^+)$ the power set of $\N^+$, define
\beqa
\label{eq:5}
f: \; 
\N^+ & \longrightarrow & \mathcal{P}(\N^+) \\
\nonumber
n & \longmapsto & \bigcup_{i\in\mathbb{N}^+} \{p_i,(p_i)^2,\ldots,(p_i)^{q_i(n)}\} ,
\eeqa
with the convention $\{p_i,(p_i)^2,\ldots,(p_i)^{q_i(n)}\} = \emptyset$ for $q_i(n) = 0$. The 
function $f$, where $f(1)=\emptyset$, maps any natural number $n>0$ to the set of its (highest) prime powers 
(naturally with exponents of at least 1), with the additional feature,
that for any prime in the factorization of $n$, all powers of this prime below the highest one are included into the set as well.

\begin{example}
{\em 
$f(18) = f(2\cdot 3^2) = \{2, 3, 3^2\}$, but $\{2, 3^2\}, \{2, 3, 3^3\}\notin f(\N^+)$.
$f(16) = f(2^4) = \{2, 2^2, 2^3, 2^4\}$, however, for instance, $\{2^4\}, \{2, 2^4\}, \{2^2, 2^4\}, \{2^3, 2^4\}$, 
$\{2, 2^2, 2^4\}, \{2^2, 2^3, 2^4\} \notin f(\N^+)$.
}
\end{example}

Clearly, $f$ is an injection with the inverse
\beqa
\label{eq:5a}
g := f^{-1}: \; 
f(\N^+) & \longrightarrow & \N^+ ,
\eeqa
which maps finite sets of prime powers (where prime powers below the maximal one are included) to their product, i.e.~to the
natural number with the corresponding prime factorization. 
More precisely, $g$ maps any member set of $f(\N^+)$ to the product of its maximal prime powers (with $g(\emptyset) = 1$),
and
\beqa
\label{eq:5b}
f: \; 
\N^+ & \longrightarrow & f(\N^+) 
\eeqa
is a bijection with inverse $g$.

\begin{example}
{\em 
$g(\{2, 3, 3^2\}) = g(f(18)) = 2 \cdot 3^2 = 18$ and $g(\{2, 2^2, 2^3, 2^4\}) = 2^4 = 16$.
}
\end{example}

\begin{lemma}
\label{lem:1}
For $m\in\mathbb{N}^+\setminus\{1\}$ and $n_j\in\mathbb{N}^+$ for $j=1,\ldots,m$,
\beqa
\label{eq:6}
f(\text{\em lcm}(n_1,\ldots,n_m)) & = & \bigcup_{j\in\{1,\ldots,m\}} f(n_j) , \\
\label{eq:7}
f(\text{\em gcd}(n_1,\ldots,n_m)) & = & \bigcap_{j\in\{1,\ldots,m\}} f(n_j) . 
\eeqa
For $m,n\in\N^+$,
\beqa
\label{eq:9wvi}
m|n & \Leftrightarrow & f(m) \subset f(n) ,\\
\label{eq:11wvi}
q = \text{\em lcm}(m,n) & \Leftrightarrow & f(q) = f(m) \cup f(n) ,\\
\label{eq:12wvi}
q = \text{\em gcd}(m,n) & \Leftrightarrow & f(q) = f(m) \cap f(n) ,\\
\label{eq:13wvi}
n \text{\em ~prime} & \Leftrightarrow & f(n) \text{\em ~inclusion-minimal in } f(\N^+) . 
\eeqa
\end{lemma}

\begin{proof}
Equations \eqref{eq:6} and \eqref{eq:7} follow in a straightforward manner from \eqref{eq:3} and \eqref{eq:4}.
Equivalence \eqref{eq:9wvi} holds since a divisor is the product of a subset of primes (prime powers) in the factorization 
of the number it divides. 
The equivalences \eqref{eq:11wvi} and \eqref{eq:12wvi} follow directly with \eqref{eq:6} and \eqref{eq:7}.
Statement \eqref{eq:13wvi} holds because of \eqref{eq:9wvi} and since $f(1) = \emptyset$.
\end{proof}

Obviously,
\beqa
g(f(m)\cup f(n)) & = & \text{lcm}(m,n) ,\\
g(f(m)\cap f(n)) & = & \text{gcd}(m,n) .
\eeqa

\begin{corollary}
\label{cor:1}
~
\begin{enumerate}
\item
$\mathcal{N}\subset\mathbb{N}^+$ is LCM-closed, respectively GCD-closed, if and only if
$f(\mathcal{N})$ is union-closed, respectively intersection-closed.
\item
$f(\N^+)$ is an order-theoretical ring.
\end{enumerate}
\end{corollary}

\begin{proof}
The first statement follows with \eqref{eq:11wvi} and \eqref{eq:12wvi}. 
The second statement follows directly from the first one, since $\N^+$ is LCM- and GCD-closed.
\end{proof}

The bijection \eqref{eq:5b}
and its inverse $g$ have thus established an isomorphism between
the strictly positive natural numbers,
equipped with the operations ``$\text{lcm}$'' and ``$\text{gcd}$'',
and the order-theoretical ring $f(\N^+)$ of, essentially, the sets of prime powers of these numbers,
equipped with the union and intersection operation.
Certainly folklore knowledge, with the explanation below Definition \ref{def:2qwe}, any finite simultaneously
LCM- and GCD-closed set $\mathcal{N}\subset\N^+$ is a complete lattice with ``$\text{gcd}$'' as the meet
operation and ``$\text{lcm}$'' as the join operation. With \eqref{eq:9wvi}, it is clear that the order is given
by smaller elements dividing larger ones.

Note that while $\N^+$ with ``$\text{lcm}$'' is a commutative monoid (or abelian monoid; the neutral element is 1,
since $\text{lcm}(n,1) = 1$),  $\N^+$ with ``$\text{gcd}$'' is no monoid, such that neither an algebraic
ring structure, nor something similar (e.g.~a ring without the negative elements) can be established on $\N^+$ with
regard to the operations given by the LCM and the GCD. 


\section{Main result}

\label{Main result}

\begin{theorem}
\label{theo:1}
Conjecture \ref{con:1} is equivalent to the union-closed sets conjecture.
\end{theorem}

\begin{proof}
$\Rightarrow:$
Assume that an arbitrary finite union-closed system, $\mathcal{S}$, with a maximal set 
$S_{\max} = \{1,2,\ldots,m\}$, $m\in\N^+$ is given. In a first step, replace all
sets of $\mathcal{S}$ with sets of corresponding prime numbers by means of replacing
any number $i\in\{1,\ldots,m\}$ with the $i$-th prime number $p_i$. For instance,
one now has $S_{\max} = \{p_1, p_2,\ldots,p_m\}$. By this procedure,
a set-theoretical isomorphism has been established between the two systems, which preserves
operations such as unions and intersections, or relationships such as inclusion. Clearly, under
this isomorphism, the system's set-related properties
such as union-closedness, or frequencies of the occurrence of elements in
member sets, remain unchanged. The elements of member sets of $\mathcal{S}$ are now prime numbers, 
so $\mathcal{S}\subset f(\N^+)$ and $g(\mathcal{S})\subset\N^+$.
Since $f$ is a bijection with inverse $g$ (cf.~\eqref{eq:5b}), $f(g(\mathcal{S})) = \mathcal{S}$.
By Corollary \ref{cor:1}, $g(\mathcal{S})$ is LCM-closed, since $\mathcal{S}$ is union-closed by assumption.
Under Conjecture \ref{con:1}, a subfamily of numbers
$\mathcal{N}\subset g(\mathcal{S})$ comprised of more than half the numbers in $g(\mathcal{S})$ 
shares a divisor larger than one, which implies that all numbers in $\mathcal{N}$ have at least one
common prime number (or even a product of primes)
in their corresponding factorizations. Thus, all members of $f(\mathcal{N})\subset\mathcal{S}$,
which are at least half of the members of $\mathcal{S}$, share this prime number 
(or these prime numbers in the corresponding product)
as an element (as elements).\\
$\Leftarrow:$
Assume that $\mathcal{N}\subset\mathbb{N}^+$ is nonempty, finite, LCM-closed, and contains at least
one element larger than 1.
By Corollary \ref{cor:1}, $\mathcal{S}:=f(\mathcal{N})$ is a finite, 
union-closed system with -- by \eqref{eq:5} -- at least one nonempty member.
Under the union-closed sets conjecture, a subfamily $\tilde{\mathcal{S}}\subset\mathcal{S}=f(\mathcal{N})$ contains
at least half of the members of  $\mathcal{S}=f(\mathcal{N})$, and they all share at least one element.
By construction, elements of member sets $S\in\mathcal{S}=f(\mathcal{N})$ are prime powers in the factorization
of $g(S)\in\mathcal{N}$. The restriction $f: \mathcal{N} \rightarrow f(\mathcal{N})$ being a bijection with inverse $g$, 
this means that at least half of the members of $\mathcal{N}$ share the same prime power, and thus have a non-trivial divisor.
\end{proof}

\begin{example}
{\em 
The two directions of the proof are illustrated.\\
$\Rightarrow:$
Using $\mathcal{S}$ of Example \ref{ex:5rtu} as an instance, the union-closed system
\beqa
\label{eq:17wer}
\mathcal{S} & = & \{\emptyset, \{1\} , \{1,2\} , \{1,2,3\}, \{4\} , \{1,4\} , \{1,2,4\}, \{1,2,3,4\} \} 
\eeqa
is, by $i \mapsto p_i$, first turned into the set-isomorphic
\beqa
\mathcal{S} & = & \{\emptyset, \{2\} , \{2,3\} , \{2,3,5\}, \{7\} , \{2,7\} , \{2,3,7\}, \{2,3,5,7\} \} ,
\eeqa
which produces the LCM-closed set (note that $210=2\cdot 3\cdot 5\cdot 7$)
\beqa
g(\mathcal{S}) & = & \{1, 2, 6, 30, 7, 14, 42, 210\} .
\eeqa
$\Leftarrow:$
Using $\mathcal{N}$ of Example \ref{ex:1} as an instance, the LCM-closed set
\beqa
\mathcal{N} 
& = & 
\{1,2,3,4,6,8,12,24\} \\
\nonumber
& = & 
\{1,2,3,2^2,2\cdot 3,2^3,2^2\cdot 3,2^3\cdot 3\}
\eeqa
is turned into the union-closed system
\beqa
\label{eq:21wer}
f(\mathcal{N}) & = & \{\emptyset, \{2\} , \{3\} , \{2,2^2\}, \{2,3\} , \{2,2^2,2^3\} , \{2,2^2,3\}, \{2,2^2,2^3,3\} \} .
\eeqa
The bijection
\beqa
1 &\longmapsto& 2 \\
\nonumber
2 &\longmapsto& 2^2 \\
\nonumber
3 &\longmapsto& 2^3 \\
\nonumber
4 &\longmapsto& 3
\eeqa
establishes, of course, a set-theoretical isomorphism between \eqref{eq:17wer} and \eqref{eq:21wer}, thus illustrating
that -- similar to names of elements being irrelevant with regard to the structure of union-closedness -- specific numbers
are irrelevant if it comes to the structure of LCM-closedness. The example therefore illustrates that the same LCM structure
and the same union-closedness structure can be represented by a multitude of sets of numbers or of families of sets.
}
\end{example}


\section{Dual version}

\label{Dual version}

A well-known dual equivalent of the union-closed sets conjecture exists (e.g.~\cite{BS}).\\

\noindent
\textbf{Intersection-closed sets conjecture.}
{\em 
For a finite intersection-closed family of sets with at least two member sets, there exists an element
shared by at most half of the member sets.\\
}

One direction of this equivalence is quickly established by considering the complementary, or dual, union-closed system
of a given intersection-closed system, which is obtained from the complements of the members of the 
intersection-closed system with respect to the set of its union. Union-closedness now follows
with De Morgan's laws. Because of the bijection between the systems, the abundant element of the 
union-closed system is the non-abundant one of the intersection-closed system.
The reverse direction is shown with a similar argument, where a union-closed system with at least
two members has to be the starting point.

This dual -- intersection-- version of Frankl's conjecture is typically used to draw the link to lattice
theory (e.g.~\cite{K}), where the meet operation is intersection, and the ordering is given by
set inclusion (see also \cite{BS}).

It should now come as no surprise, that an equivalent dual of Conjecture \ref{con:1} exists.

\begin{conjecture}
\label{con:3}
For a GCD-closed finite set of non-zero natural numbers that contains at least two elements, 
one of its members has a prime power that is not a prime power of more than half of the members.
\end{conjecture}

A prime power as in Conjecture \ref{con:3} will be called non-abundant.

\begin{example}
\label{ex:9asx}
{\em
Revisit the GCD-closed set $\mathcal{N}=\{1,2,3,4,6,8,12\}$ of Example \ref{ex:0}, which has seven elements.
The prime powers $2^2$, $2^3$, and $3$ are non-abundant.
}
\end{example}

\begin{corollary}
\label{cor:2}
Conjecture \ref{con:3} is equivalent to the intersection-closed sets conjecture, and thus to
the union-closed sets conjecture, and to Conjecture \ref{con:1}.
\end{corollary}

\begin{proof}
The proof of equivalence for Conjecture \ref{con:3} and the intersection-closed sets conjecture follows 
in very close analogy to the proof of Theorem \ref{theo:1}, with the obviously necessary replacements of 
``union'' by ``intersection'', ``LCM'' by ``GCD'', and ``Conjecture \ref{con:1}'' by ``Conjecture \ref{con:3}''.
\end{proof}

While Conjecture \ref{con:3} has been established as the dual of Conjecture \ref{con:1}, it may still
be of interest to look at e.g.~the specific LCM-closed dual of a given GCD-closed system.
In order to do this, define for any finite $\mathcal{N}\in\N^+$, and with 
\beqa
\hat n & := & \text{lcm}(\mathcal{N}) ,
\eeqa
a function $h$ on $\mathcal{N}$ by
\beqa
\label{eq:17}
h: \;
n = \prod_{i\in\mathbb{N}^+}(p_i)^{q_i(n)} & \longmapsto & \prod_{i\in\mathbb{N}^+}(p_i)^{q_i(\hat n) - q_i(n)} ,
\eeqa
and denote
\beqa
\label{eq:16}
\mathcal{N}^* 
& = & 
h(\mathcal{N}) .
\eeqa

\begin{example}
\label{ex:10asx}
{\em
For the GCD-closed set $\mathcal{N}=\{1,2,3,4,6,8,12\}$ of Example \ref{ex:9asx} (and Ex.~\ref{ex:0}),
$\text{lcm}(\mathcal{N}) = 24 = 2^3\cdot 3$, such that $\mathcal{N}^* = h(\mathcal{N}) = \{2,3,4,6,8,12,24\}$.
For instance,
\beqa
h(1) & = & 2^{3-0}\cdot 3^{1-0} \; = \; 24 ,\\
h(12) & = & h(2^2\cdot 3) \; = \;  2^{3-2}\cdot 3^{1-1} \; = \; 2 .
\eeqa
}
\end{example}

\begin{lemma}
\label{lem:2}
For any finite $\mathcal{N}\in\N^+$, it holds that $\mathcal{N}^*=h(\mathcal{N})\subset\N^+$ and 
\beqa
h: \; \mathcal{N} & \longrightarrow & \mathcal{N}^* 
\eeqa
is a bijection.
\end{lemma}

\begin{proof}
For all $n\in\mathcal{N}$, the exponents $q_i(\hat n) - q_i(n)$ in \eqref{eq:17} 
are non-negative because of \eqref{eq:3}. Thus, \eqref{eq:16} is a subset of $\N^+$.
Since $q$ of \eqref{eq:1} is an injection,
$h: \mathcal{N} \rightarrow \mathcal{N}^*$
is a bijection. 
\end{proof}

\begin{proposition}
\label{prop:3qwr}
~
\begin{enumerate}
\item
Let $\mathcal{N}\in\N^+$ be a GCD-closed set adhering to Conjecture \ref{con:3}.
Then $\mathcal{N}^*=h(\mathcal{N})$ is an LCM-closed set adhering to Conjecture \ref{con:1}.
\item
Let $\mathcal{N}\in\N^+$ be an LCM-closed set with at least two members adhering to Conjecture \ref{con:1}.
Then $\mathcal{N}^*=h(\mathcal{N})$ is a GCD-closed set adhering to Conjecture \ref{con:3}.
\end{enumerate}
\end{proposition}

\begin{proof}
1. 
By Lemma \ref{lem:2}, $\mathcal{N}^*$ is finite and has more than one element (since it has
as many as $\mathcal{N}$, which has at least two), thus containing an element larger than 1.
Consider now two elements $m^*, n^*\in\mathcal{N}^*$ with pre-images $m, n\in\mathcal{N}$ under $h$. 
By \eqref{eq:17},
\beqa
q_i(m^*) 
& = & 
q_i(h(m)) \; = \; q_i(\hat n) - q_i(m) , \\
q_i(n^*) 
& = & 
q_i(h(n)) \;\; = \; q_i(\hat n) - q_i(n) .
\eeqa
Since $ \text{gcd}(m,n)\in\mathcal{N}$, one obtains with \eqref{eq:3} and \eqref{eq:4} that
\beqa
\text{lcm}(m^*, n^*) 
& = & 
\prod_{i\in\mathbb{N}^+}(p_i)^{\max\{q_i(m^*),q_i(n^*)\}} \\
\nonumber
& = & 
\prod_{i\in\mathbb{N}^+}(p_i)^{q_i(\hat n) - \min\{q_i(m), q_i(n)\}} \\
\nonumber
& = & 
h\left(
\prod_{i\in\mathbb{N}^+}(p_i)^{\min\{q_i(m), q_i(n)\}}
\right) \\
\nonumber
& = & 
h(\text{gcd}(m,n)) \; \in \; \mathcal{N}^* ,
\eeqa
proving that $\mathcal{N}^*$ is LCM-closed.
Let $(p_j)^{q^*_j}$ be a prime power that occurs (as in: divides)
 in at least one member of $\mathcal{N}$, but not in more than half of them. 
This means that for at least half of the members $n\in\mathcal{N}$, it holds that $q_j(n) < q^*_j$. But this implies
that for at least half of the members $h(n)\in\mathcal{N}^*$, it holds that $q_j(h(n)) > q(\hat n) -q^*_j \geq 0$,
meaning that $\mathcal{N}^*$ has with $(p_j)^{q(\hat n) -q^*_j+1}$
an abundant prime power and, therefore, an abundant divisor.\\
2.
Up to the point, where $\mathcal{N}^*$ emerges as GCD-closed, the proof of the second statement follows in close analogy
the one of the first statement.
Let now $(p_j)^{q^*_j}$ be a prime power in at least half the members of $\mathcal{N}$, but not in all of them. 
At this point, the question arises, if such a prime power exists, as it -- at first sight  -- could be that all abundant prime powers
occurred in all members. However, that this cannot be the case follows from the easy to check fact that with
$\mathcal{N}$, also $\mathcal{N}/\text{gcd}(\mathcal{N})$ is LCM-closed, and $n \mapsto n/\text{gcd}(\mathcal{N})$ is
an LCM-consistent bijection between the two sets. Since Conjecture \ref{con:1} then also applies to
$\mathcal{N}/\text{gcd}(\mathcal{N})$, it is clear that there would be the required type of prime power in this set, which,
by multiplication with the prime power of the same basis (prime) in $\text{gcd}(\mathcal{N})$, would deliver the $(p_j)^{q^*_j}$
with the required properties.
This means that for at least half of the members $n\in\mathcal{N}$, it holds that $q_j(n) \geq q^*_j$, implying 
that for at least half of the members $h(n)\in\mathcal{N}^*$, it holds that $q_j(h(n)) \leq q(\hat n) - q^*_j$.
However, since  $(p_j)^{q^*_j}$ is not a prime power in all members of $\mathcal{N}$, the prime power
$(p_j)^{q(\hat n) -q^*_j +1}$ does exist in (as in: divides) members of $\mathcal{N}^*$, but in at most half of them.
\end{proof}

\begin{example}
\label{ex:11asx}
{\em
This is an example for the first statement in Prop.~\ref{prop:3qwr}.
The GCD-closed set $\mathcal{N}=\{1,2,3,4,6,8,12\}$ of Example \ref{ex:0} and \ref{ex:9asx} has seven elements,
and the prime powers $2^2$, $2^3$, and $3$ are non-abundant. Moreover, $\text{lcm}(\mathcal{N}) = 24 = 2^3\cdot 3$.
In Example \ref{ex:10asx}, $\mathcal{N}^* = h(\mathcal{N}) = \{2,3,4,6,8,12,24\}$ was determined.
It can easily be checked that $\mathcal{N}^*$ is LCM-closed, and that it has the abundant prime powers (divisors)
$2^2 = 2^{3-2+1}$, $2 = 2^{3-3+1}$, and $3 = 3^{1-1+1}$.
}
\end{example}

\begin{definition}
Let $\mathcal{N}\in\N^+$ be a non-empty finite GCD-closed (LCM-closed) set.
Then $\mathcal{N}^*$ as in \eqref{eq:16} is called its LCM-closed (GCD-closed) dual.
\end{definition}

Note that \eqref{eq:3} and \eqref{eq:4} provide the
minimum and the maximum of the occurring power exponents of the prime $p_i$ in the members of $\mathcal{N}$
as $q_i(\text{gcd}(\mathcal{N}))$ and $q_i(\text{lcm}(\mathcal{N}))$. It follows with
\eqref{eq:17} that the minimum and the maximum of the occurring power exponents of the prime $p_i$ in the 
members of $\mathcal{N}^*$ are $0$ and $q_i(\text{lcm}(\mathcal{N})) - q_i(\text{gcd}(\mathcal{N}))$.
As was pointed out in the proof of Proposition \ref{prop:3qwr} for GCD-closed $\mathcal{N}$, observe that for
LCM-closed (GCD-closed) $\mathcal{N}$, the set $\mathcal{N}/\text{gcd}(\mathcal{N})$ is again LCM-closed (GCD-closed), and 
\beqa
(\mathcal{N}/\text{gcd}(\mathcal{N}))^*
& = &
\mathcal{N}^* , 
\eeqa
since, on the right hand side of \eqref{eq:17},
the division of $\mathcal{N}$ by $\text{gcd}(\mathcal{N})$ simply reduces $q_i(\hat n)$ and $q_i(n)$ both by
$q_i(\text{gcd}(\mathcal{N}))$, thus -- in total -- resulting in no change to $\mathcal{N}^*$. 
Therefore, 
\beqa
(\mathcal{N}^*)^* 
& = &
\mathcal{N}
\eeqa
holds if and only if $\text{gcd}(\mathcal{N})=1$, and, thus, $\text{lcm}(\mathcal{N}) = \text{lcm}(\mathcal{N}^*)$.

\begin{example}
\label{ex:12asx}
{\em
In Example \ref{ex:11asx}, $\text{gcd}(\mathcal{N})=\text{gcd}(\mathcal{N}^*)=1$ and $\mathcal{N} = (\mathcal{N}^*)^*$.
However, in Example \ref{ex:4asx}, 
\beqa
\mathcal{N}
& = & 
\{6,10,14,30,42,70,210\} ,
\eeqa
and $\mathcal{N}$ is LCM-closed, but not GCD-closed, 
$\text{gcd}(\mathcal{N}) = 2$, and $\text{lcm}(\mathcal{N}) = 210 = 2\cdot 3\cdot 5\cdot 7$. This implies
\beqa
\mathcal{N}
& = & 
\{6,10,14,30,42,70,210\} ,\\
\mathcal{N}/\text{gcd}(\mathcal{N})
& = &
\{3,5,7,15,21,35,105\} ,\\
\mathcal{N}^* \; = \; (\mathcal{N}/\text{gcd}(\mathcal{N}))^*
& = &
\{1,3,5,7,15,21,35\} ,
\eeqa
and, since $\text{lcm}(\mathcal{N}^*) = 105$,
\beqa
(\mathcal{N}^*)^*
& = &
\mathcal{N}/\text{gcd}(\mathcal{N})
\; = \;
\{3,5,7,15,21,35,105\} .
\eeqa
}
\end{example}


\section{Examples of scope}

\label{Examples of scope}

The sum of the prime power exponents of a natural number $n\in\N^+$ is given by
\beqa
\sigma_{\text{PPE}}(n)
& = & 
\sum_{i\in\mathbb{N}^+} q_i(n) .
\eeqa
With the findings so far, and returning to the list at the end of Section \ref{Frankl's union-closed sets conjecture},
Conjecture \ref{con:1} for instance holds for an LCM-closed finite set $\mathcal{N}\subset\N^+$
that contains at least one element larger than 1, if:
\begin{enumerate}
\item
There is an $n\in\mathcal{N}$ with $n\neq 1$, such that the prime factorization of $n$ has a maximum of two prime factors.
\item
$\sigma_{\text{PPE}}(\text{lcm}(\mathcal{N})/\text{gcd}(\mathcal{N})) \leq 12$.
\item
$\#\mathcal{N} \leq 50$.
\item
$\#\mathcal{N} \geq \frac{2}{3}2^{\sigma_{\text{PPE}}(\text{lcm}(\mathcal{N})/\text{gcd}(\mathcal{N}))}$.
\item
$\#\mathcal{N} \leq 2(\sigma_{\text{PPE}}(\text{lcm}(\mathcal{N})/\text{gcd}(\mathcal{N})))$, if $\mathcal{N}$ is such that for 
any two distinct prime powers (for instance, $2^2$ and $2^3$ are two distinct prime powers in $24$)
$p'$ and $p''$ in $\text{lcm}(\mathcal{N})/\text{gcd}(\mathcal{N})$,
there exist $m,n\in\mathcal{N}/\text{gcd}(\mathcal{N})$ such that $p'|m$, $p''|n$, but neither $p''|m$, nor $p'|n$.
\end{enumerate}


\section{Conclusion}

The long-standing union-closed sets conjecture is known for its applications in graph and lattice theory.
This note presents a link of the conjecture to another mathematical field: multiplicative number theory.
It would be nice if the here presented equivalent number theoretical conjecture or methods turned out
to be fruitful in any way.



\end{document}